%% file: paper_v3.3.tex
\documentclass[twoside, 11pt]{article}

\usepackage{graphicx, subcaption}
\usepackage[margin=1in, footnotesep=0.5in]{geometry}
\usepackage[sort&compress, numbers]{natbib} 
\usepackage{amsfonts, amsmath, amssymb, amsthm}
\usepackage[shortlabels]{enumitem}
\usepackage{booktabs}
\usepackage{mathtools}
\mathtoolsset{showonlyrefs}

\usepackage{color}

\usepackage{hyperref}
\definecolor{customdarkred}{RGB}{150,0,0}
\definecolor{customdarkgreen}{RGB}{0,150,0}
\definecolor{customdarkblue}{RGB}{0,0,150}
\hypersetup{colorlinks=true, linkcolor=customdarkred, citecolor=customdarkgreen, urlcolor=customdarkblue}
\usepackage{url}
\usepackage[linesnumbered, ruled]{algorithm2e}
\usepackage{authblk}

\input notation.tex

\usepackage[normalem]{ulem}

\title{Stability of Sequential Lateration and of Stress Minimization\\ in the Presence of Noise}
\author[1,2]{Ery Arias-Castro}
\author[1]{Siddharth Vishwanath} 
\affil[1]{\small Department of Mathematics, University of California, San Diego} 
\affil[2]{\small Halıcıoğlu Data Science Institute, University of California, San Diego}
\date{}

\begin{document}
\maketitle
\thispagestyle{empty}

\begin{abstract}
Sequential lateration is a class of methods for multidimensional scaling where a suitable subset of nodes is first embedded by some method, e.g., a clique embedded by classical scaling, and then the remaining nodes are recursively embedded by lateration. A graph is a lateration graph when it can be embedded by such a procedure. We provide a stability result for a particular variant of sequential lateration. We do so in a setting where the dissimilarities represent noisy Euclidean distances between nodes in a geometric lateration graph. 
We then deduce, as a corollary, a perturbation bound for stress minimization. 
To argue that our setting applies broadly, we show that a (large) random geometric graph is a lateration graph with high probability under mild conditions, extending a previous result of Aspnes et al (2006). 
\end{abstract}

\section{Introduction}

In multidimensional scaling (MDS), we are provided with some pairwise dissimilarities between a number of items, and the general goal is to embed these items as points in a Euclidean space of given dimension in such a way that the resulting Euclidean distances reproduce, as faithfully as possible, the dissimilarities.
MDS is a well-studied problem in psychometrics \cite{borg2005modern}, mathematics and computer science (embedding of metric spaces) \cite{blumenthal1953theory}, in optimization (Euclidean distance matrix completion) \citep{laurent2001matrix}, and engineering (sensor network localization) \cite{mao2007wireless}, and it is an integral part of multivariate statistical analysis \cite{seber2009multivariate, MVA} and unsupervised machine learning \cite{hastie2009elements}. MDS is closely related to the problem of graph drawing \cite{battista1998graph, klimenta2012extending}.

\subsection{Setting} 
\label{sec:setting}
More formally, we are given an undirected graph $\cG = (\cV, \cE)$, with node set $\cV = [n] := \{1, \dots, n\}$ and edge set $\cE \subset \cV \times \cV$, together with non-negative weights on the edges. The weight on $(i, j) \in \cE$ is referred to as the dissimilarity between $i$ and $j$, and denoted $d_{ij}$. The (possibly incomplete) matrix $D = (d_{ij})$ stores these dissimilarities. Based on this information, we seek to embed the nodes into a Euclidean space of given dimension, denoted $p$, as accurately as possible. Specifically, we seek a configuration $y_1, \dots, y_n \in \bbR^p$ such that $\|y_i - y_j\| \approx d_{ij}$ for all or most $(i,j) \in \cE$. 
A notion of stress, for example, the s-stress of Takane~et~al.~\cite{takane1977nonmetric} defined as
\begin{align}
\label{stress}
\sum_{(i,j) \in \cE} \big(\|y_i - y_j\|^2 - d_{ij}^2\big)^2,
\end{align}
offers a way to quantify the accuracy of the embedding. 
Throughout, the dimension $p$ will be assumed given and $\|\cdot\|$ will denote the Euclidean norm in $\bbR^p$.

We say that the graph is realizable (in dimension $p$) if there is a point set $y_1, \dots, y_n \in \bbR^p$ such that $\|y_i - y_j\| = d_{ij}$ for all $(i,j) \in \cE$, or in words, if there is a configuration with zero stress.
In this paper we are most interested in the noisy realizable situation in which
\begin{align}
\label{setting}
d_{ij}^2 = \|x_i - x_j\|^2 + \eps_{ij}, \quad (i,j) \in \cE,
\end{align}
where $\{x_1, \dots, x_n\} \in \bbR^p$ will be referred to as the {\em latent configuration}\footnote{Note that the latent configuration is only determined up to a rigid transformation, as we do not assume that any anchor is available. However, this duplicity does not cause any trouble.}
 and $\{\eps_{ij}: (i,j) \in \cE\}$ represents measurement noise, possibly stochastic.
This additive noise model is considered in a number of places, e.g., \cite{anderson2010formal, javanmard2013localization, li2020central}. It includes, as a special case, the following multiplicative noise model 
\begin{align}
\label{setting mult}
d_{ij} = (1+\eta_{ij}) \|x_i - x_j\|, \quad (i,j) \in \cE,
\end{align}
by simply setting $\eps_{ij} = 2 \eta_{ij} \|x_i - x_j\| + \eta_{ij}^2 \|x_i - x_j\|^2$ in \eqref{setting}. Although the model \eqref{setting} is in principle completely general, in our results we will bound the error terms. It is possible to study the problem under more general assumptions as recently done in \cite{kroshnin2022infinite, lim2022classical}, but the model above is most appropriate for our purposes as will become clear below. 

\subsection{Methods}
\label{sec:methods}
A wide array of approaches have been proposed to tackle this problem, starting with {\em classical scaling}, the oldest and still the most popular method, proposed by Togerson \cite{torgerson1958theory, torgerson1952multidimensional} and further developed by Gower \cite{gower1966some}, with roots in a mathematical inquiry by Young \& Householder \cite{young1938discussion} into necessary and sufficient conditions ``for a set of numbers to be the mutual distances of a set of real points in Euclidean space'' --- to quote the abstract of their cornerstone paper.
Kruskal \cite{kruskal1964nonmetric, kruskal1964multidimensional} formulated the problem as minimizing a notion of stress that he suggested for that purpose --- same as \eqref{stress} but without the squares inside the brackets. Many other optimization approaches have been tried, including second order methods \cite{kamada1989algorithm}, as well as other Newton and quasi-Newton variant procedures \cite{kearsley1998solution, glunt1993molecular}; augmentation and majorization \cite{de1975alternating, heiser1988multidimensional}, which include the SMACOF algorithm \cite{de1977applications, de2009multidimensional, mair2022more}, itself closely related to the fixed point iteration approach of Guttman \cite{guttman1968general}; incremental and multigrid approaches \cite{cohen1997drawing, williams2004steerable, bronstein2006multigrid}; divide-and-conquer or patch-stitching algorithms \cite{yang2006fast, tzeng2008multidimensional, koren2005patchwork, cucuringu2012sensor, singer2008remark, zhang2010rigid, drusvyatskiy2017noisy, krislock2010explicit, shang2004improved, hendrickson1995molecule}; semidefinite programming (SDP) formulations where the constraint on the embedding dimension is removed \cite{alfakih1999solving, biswas2006semidefinite, biswas2006semidefinite_based, javanmard2013localization, weinberger2006graph, so2007theory, drusvyatskiy2017noisy, cayton2006robust}; and the completion of the dissimilarity matrix by graph distances before applying a method like classical scaling \cite{kruskal1980designing, priyantha2003anchor, shang2003localization, niculescu2003dv}. 
See the book by Borg \& Groenen \cite{borg2005modern} and the PhD thesis of Klimenta \cite[Ch 2, 3]{klimenta2012extending} for partial reviews of the literature.

\subsubsection*{Sequential lateration} 
We place our attention on {\em sequential lateration}, which is an approach in which a suitable subset of nodes is first embedded by some method --- e.g., a clique embedded by classical scaling --- and then the remaining nodes are recursively embedded by lateration \cite{kearsley1998solution, aspnes2006theory, fang2009sequential, eren2004rigidity, laurent2001polynomial, bakonyi1995euclidian, grone1984positive}.
 
{\em Lateration} is the problem of locating a point based on its (possibly inaccurate) distances to a set of given points often referred to as anchors, beacons or landmarks. The problem is known under different names, including `trilateration' or `multilateration', or simply `lateration', in engineering \citep{chrzanowski1965theoretical, glossary1994, fang1986trilateration, yang2009indoor, savvides2001dynamic, navidi1998statistical}, while `external unfolding' is favored in psychometrics \cite{borg2005modern, carroll1972individual}.  

Aspnes~et~al.~\cite{aspnes2006theory}, developing ideas already present in their prior work \cite{eren2004rigidity}, introduce\footnote{They use `trilateration graph' and `trilaterative ordering' as they focus on the case of dimension $p = 2$.} the notion of {\em lateration graph} (in dimension $p$), which they define as a graph with $n \ge p+1$ vertices that admits an ordering of its vertices, say $v_1, \dots, v_n$, such that the subgraph induced by $v_1, \dots, v_{p+1}$ is complete and, for each $j > p+1$, $v_j$ is connected to at least $p+1$ vertices among $v_1, \dots, v_{j-1}$ --- they call this a {\em laterative ordering} (in dimension $p$).
They show in \cite[Th 10]{aspnes2006theory} that the problem of \secref{setting} is solvable in polynomial time by sequential lateration in the realizable situation \eqref{setting} (with $\eps_{ij} \equiv 0$) when the latent points $x_1, \dots, x_n$ are in general position and the graph $(\cV, \cE)$ is a lateration graph. 
We say that a configuration is in general position if any $(p+1)$-tuple from the configuration spans the entire space.

\subsection{Contribution and content}
Our main contribution is establishing a perturbation bound for sequential lateration. Such a bound helps us understand how the performance of a method degrades with the presence of noise. While, as already mentioned, sequential lateration is exact in the realizable setting when the latent points are in general position and the graph is a lateration graph, our study provides an understanding of how the method behaves in the noisy realizable setting \eqref{setting}. 

As our second contribution, we use our perturbation bound for sequential lateration to derive a perturbation bound for stress minimization in the same setting of a lateration graph. Although stress minimization is not an algorithm per se, we show that the set of configurations that minimize the stress \eqref{stress} is stable in the presence of noise in the sense that any minimizing configuration is within a distance (up to rigid transformations) to the latent configuration in \eqref{setting} controlled in terms of the amount of noise. 

Only a few perturbation bounds exists in the MDS literature. For classical scaling, some partial results were developed early on by Sibson \cite{sibson1979studies} and later revisited by de Silva \& Tenenbaum \cite{de2004sparse}, but a true perturbation bound was only established recently in \cite{arias2020perturbation}, where perturbation bounds for the completion by graph distance method of Kruskal \& Seery \cite{kruskal1980designing} and the SDP method of Weinberger~et~al.~\cite{weinberger2006graph} --- in the context of manifold learning in the form of isomap \cite{Tenenbaum00ISOmap} and maximum variance unfolding \cite{weinberger2006introduction} --- were also obtained.  
Similarly, some perturbative results were derived in \cite{de2004sparse} for lateration, but a true perturbation bound was only achieved in \cite{arias2020perturbation} (to our knowledge).
Moore~et~al.~\cite{moore2004robust}, inspired by the earlier work of Eren~et~al.~\cite{eren2004rigidity}, propose a method for sequential trilateration and carry out a very limited mathematical analysis, restricting themselves to analyzing the probability of a gross error or `flip' in one trilateration step. We are not aware of any other results.

Our perturbation bound for sequential lateration was perhaps anticipated by Anderson~et~al.~\cite{anderson2010formal}, who discuss this as an open question in their last section: 

\begin{quote}
``An important problem, linked but separate from the one treated in this paper, is how (numerically) to solve the minimization problem. The corresponding problem in the noiseless case is how to perform localization. For a localization problem to be solvable in polynomial time, it is generally necessary that some special structure holds for the graph; for example, in the case of trilateration graphs, localization can be done in linear time with suitable anchors \cite{anderson2009graphical}. We would expect, although we have no formal proof, that such geometries will also be important in ensuring that a noisy localization problem is computationally tractable.''
\end{quote}

Perturbation bounds and, more generally, a better understanding of the MDS problem under noise, were open problems discussed at length by Mao~et~al.~\cite{mao2007wireless} in their well-cited review paper of the engineering literature on the topic.
We leverage our perturbation bound for sequential lateration to obtain another result that contributes to that endeavor: we show, in the same context, that any configuration that minimizes the stress \eqref{stress} is necessarily close to the latent configuration. In doing so, we recover a result of Anderson~et~al.~\cite{anderson2010formal} in the special case of a lateration graph.  

Although not all graphs are lateration graphs, the setting covers the main stochastic model used in the literature, that of a random geometric graph. While this was already known to Eren~et~al.~\cite{eren2004rigidity} and Aspnes~et~al.~\cite{aspnes2006theory}, as our third contribution, we provide a much more general result, showing that a large random geometric graph is a lateration graph with high probability under very mild assumptions on the underlying sampling distribution. 

The remainder of the paper is organized as follows.
In \secref{sequential lateration}, we derive a perturbation bound for sequential lateration. 
In \secref{stress minimization}, we obtain as a corollary a perturbation bound for stress minimization. This is placed in the broader context of {\em rigidity theory}. 
In \secref{rgg}, we provide rather mild conditions under which a large random geometric graph is a lateration graph with high probability. 
Some numerical experiments meant to illustrate the theory are presented in \secref{numerics}. 
And \secref{discussion} is a discussion section.

\section{Perturbation bound for sequential lateration}
\label{sec:sequential lateration}

The particular variant of sequential lateration that we work with is based on classical scaling and what we call {\em classical lateration}, a method for lateration that was originally proposed by Gower~\cite{gower1966some} and later rediscovered by de~Silva~\&~Tenenbaum~\cite{de2004sparse}, and is the analog of classical scaling for the lateration problem.

The procedure works as follows.
For each $(p+1)$-tuple of nodes within $\cV = [n]$, if it is complete, meaning that the $(p+1)$-tuple forms a clique, we embed it by classical scaling; we then recursively embed by classical lateration any node that is neighbor to at least $p+1$ nodes that have already been embedded. We can think of two main variants: in the `first' variant, we stop at the first full embedding achieved in this manner; in the `best' variant, we go through all full embeddings and select the one with smallest stress \eqref{stress}. Both variants run in polynomial time, although the `best' variant is prohibitively expensive to run in practice, having a complexity of order {$\asymp n^{p+2}$} since there are $(p+1)$-tuples to go through and, for each of them, running the sequential lateration has complexity $O(n)$. 
Our perturbation bound applies to either variant, and any other variant `in between'.  

\begin{theorem}
\label{thm:sequential lateration}
In the context of \secref{setting}, consider a noisy realizable situation as in \eqref{setting} in which the network structure $(\cV, \cE)$ is a lateration graph and the latent configuration is in general position. Then, there are $\sigma > 0$ and $A>0$, continuous in $(\|x_i-x_j\| : (i,j) \in \cE)$, such that, if $\sum_{(i,j) \in \cE} \eps_{ij}^2 \le \sigma^2$, sequential lateration outputs an embedding $y_1, \dots, y_n$ satisfying 
\begin{align}
\label{lateration1}
\min_{g} \sum_{i \in [n]} \|y_i - g(x_i)\|^2 \le A {\sum_{(i,j) \in \cE} \eps_{ij}^2},
\end{align}
where the minimization is over the rigid group of transformations of $\bbR^p$.
\end{theorem}

The proof of \thmref{sequential lateration} occupies the rest of this section. It uses two perturbation bounds from \cite{arias2020perturbation}, one for classical scaling (\lemref{CS}) and one for classical lateration (\lemref{lateration}). 
Let the width of a point set be the the minimum distance between two parallel hyperplanes that completely enclose the point set. Recall that we assume the embedding dimension $p$ to be given.
We start with a stability result for classical scaling.\footnote{The definitions of diameter and width are different in \cite{arias2020perturbation}.}

\begin{lemma}[Corollary 2 in \cite{arias2020perturbation}]
\label{lem:CS}
Consider a configuration $x_1, \dots, x_m \in \bbR^p$ with diameter $\rho$ and width $\omega > 0$, and a complete set of dissimilarities $(d_{ij})$, and define  
\begin{align}
\label{CS1}
\eta^4 = \sum_{1 \le i < j \le m} \big(d_{ij}^2 - \|x_i - x_j\|^2\big)^2.
\end{align}
There is a constant $C_1 \ge 1$ such that, if $\eta \le \sqrt{m}(\omega/C_1)$, classical scaling with input dissimilarities $(d_{ij})$ returns a configuration $y_1, \dots, y_m$ satisfying
\begin{align} \label{mds1}
    {\min_{g} \sum_{i \in [m]} \|y_i - g(x_i)\|^2 \le \frac{C_1\rho^2}{m\omega^4} \cdot {\eta^4}},
\end{align}
where the minimization is over the rigid group of transformations.
\end{lemma}

\renewcommand{\tilde}{\widetilde}

Next is a stability result for classical lateration --- the lateration method that we consider --- where stability is considered with respect to noise both at the level of the dissimilarities and at the level of the landmarks. 
In the statement, $x_1, \dots, x_m$ play the role of landmarks and $x$ is the unknown point to be recovered; $y_1, \dots, y_m$ should be seen as noisy versions of $x_1, \dots, x_m$, and $d_1, \dots, d_m$ should be seen as noisy versions of $\|x-x_1\|, \dots, \|x-x_m\|$.

\begin{lemma}[Corollary 3 in \cite{arias2020perturbation}]
\label{lem:lateration}
Consider a configuration $x_1, \dots, x_m \in \bbR^p$  with diameter $\rho$ and width $\omega > 0$, and an arbitrary point $x \in \bbR^p$. 
Let $y_1, \dots, y_m \in \bbR^p$ be another configuration and let $d_1, \dots, d_m$ be set of dissimilarities, and define 
\begin{align}
\nu^2 = \sum_{i = 1}^m \|y_i - x_i\|^2, \quad\text{and}\quad
\zeta^4 = \sum_{i = 1}^m \big(d_{i}^2 - \|x - x_i\|^2\big)^2.
\end{align}
There is a constant $C_2 \ge 1$ such that, if $\nu \le \omega/C_2$, classical lateration with inputs $y_1, \dots, y_m$ and $d_1, \dots, d_m$ outputs an embedding $y$ satisfying 
\begin{align}
{\|y-x\|^2 \le \frac{C_2}{\omega^2} \bigg(\rho^2 \nu^2 + \frac{\zeta^4}{m}\bigg).}
\end{align}
\end{lemma}

\begin{proof}[Proof of \thmref{sequential lateration}]

Assume without loss of generality that $(1, \dots, n)$ is already a laterative ordering. Therefore, for $k \ge p+1$, there is $I_{k} \subset [k]$ such that {$|I_{k}| \ge p+1$} and $(i,k+1) \in \cE$ for all $i \in I_{k}$. Note that $I_{p+1} = [p+1]$.
Let $\rho_k$ and $\omega_k$ denote the diameter and width of $\{x_i: i \in I_k\}$, 
define
\[\alpha := \max_k (\rho_k/\omega_k)^2, \qquad \underline{\omega} := \min_k \omega_k,\]
and set
\begin{align}
\nu_{k}^2 
:= \sum_{i \in [k]} \|y_i - x_i\|^2, \qquad 
\zeta_{k}^4 
:= \sum_{i \in I_k} \big(d_{k+1, i}^2 - \|x_{k+1} - x_{i}\|^2\big)^2
= \sum_{i \in I_k} \eps_{k+1, i}^2. 
\end{align}
Define $\cE_k := \big\{(i,j) \in \cE: i, j \in [k]\big\} = \cE \cap [k]^2$,
and set 
\[\eta_k^4 := \sum_{(i,j) \in \cE_{k}} \eps_{ij}^2, \qquad \eta^4 := \sum_{(i,j) \in \cE} \eps_{ij}^2.\]

{
Notice that $\cE_{k+1} = \cE_{k} \cup \{(i, k+1) : i \in I_k\}$ and $\cE_{k} \subseteq \cE$ for all $k \in [n]$; therefore, we have the following bounds which, although obvious, will be used multiple times
$$
\eta^4_{k+1} = \eta^4_k + \zeta_k^4, \qquad \eta_k \le \eta, \qquad \zeta_k \le \eta.
$$
}
We first apply classical scaling to $(d_{ij}: i,j \in [p+1])$, which we can do since these dissimilarities are available because of the assumption that $(1, \dots, n)$ is a laterative ordering. Let $y_1, \dots, y_{p+1}$ be the output.

Assuming $\eta$ is small enough that $\eta^4 \le B_{p+1} := {(p+1)^2}(\omega_{p+1}/C_1)^4$, we have $\eta_{p+1} \le {\sqrt{p+1}}(\omega_{p+1}/C_1)$, so that we may apply \lemref{CS} to get
\begin{align}
\sum_{i \in [p+1]} \|y_i - g(x_i)\|^2 \le C_1 \, \frac{\rho_{p+1}^2}{(p+1)\omega_{p+1}^4} \cdot \eta_{p+1}^4,
\end{align}
for some rigid transformation $g$.
Henceforth, we assume that $g$ is the identity transformation, which we can do without loss of generality, so that
\begin{align}
\label{sequential lateration_proof1}
\nu_{p+1}^2 \le A_{p+1} \eta_{p+1}^4, \quad \text{where } A_{p+1} := \frac{C_1\alpha}{(p+1)\underline{\omega}^2}.
\end{align}

With $y_1, \dots, y_{p+1}$ computed by classical scaling, we now compute $y_{p+2}$ by classical lateration. 
If $\eta$ is small enough that {\smash{$\eta^4 \le \min\{B_{p+1}, B_{p+2}\}$ where $B_{p+2} := A_{p+1}^{-1} (\omega_{p+1}/C_2)^{2}$}}, from \eqref{sequential lateration_proof1} we have $\nu_{p+1} \le \omega_{p+1}/C_2$, so that we may apply \lemref{lateration} to get
{
\begin{align}
\|y_{p+2} - x_{p+2}\|^2 
&\le \frac{C_2}{\omega_{p+1}^2} \bigg(\rho_{p+1}^2 \nu_{p+1}^2 + \frac{\zeta_{p+1}^4}{|I_{p+1}|}\bigg) \\
&\le {C_2} \bigg(\alpha A_{p+1}\eta_{p+1}^2 + \frac{\zeta_{p+1}^4}{(p+1)\underline{\omega}^2}\bigg)
\end{align}
where the last inequality follows by using \eqref{sequential lateration_proof1} along with the bounds $(\rho_{p+1}^2/\omega_{p+1}^2) \le \alpha$, $\omega_{p+1} \ge \underline{\omega}$ and $|I_{k}| \ge p+1$. This implies that
\begin{align}
\nu_{p+2}^2
&= \nu_{p+1}^2 + \|y_{p+2} - x_{p+2}\|^2 \\
&\le A_{p+1} \eta_{p+1}^4 + {C_2} \bigg(\alpha A_{p+1}\eta_{p+1}^2 + \frac{\zeta_{p+1}^4}{(p+1)\underline{\omega}^2}\bigg) \\
&\le (1+C_2\alpha)A_{p+1} \eta_{p+1}^4 + \frac{C_2}{(p+1)\underline{\omega}^2}\zeta_{p+1}^4 \\
&\le A_{p+2} \: (\eta_{p+1}^4 + \zeta_{p+1}^4)\\
&= A_{p+2} \: \eta_{p+2}^4,
\end{align}
where, from \eqref{sequential lateration_proof1},
\begin{align}\label{eq:Ap2}
A_{p+2} 
\le \max\bigg\{(1+C_2\alpha)A_{p+1},  \frac{C_2}{(p+1)\underline{\omega}^2}\bigg\} 
= \frac{(1+C_2\alpha)}{(p+1)\underline{\omega}^2}\max\bigg\{C_1\alpha,  \frac{C_2}{(1+C_2\alpha)}\bigg\}. 
\end{align}
}

This can be carried on for all $k$. To formalize this, we use induction. Suppose that for some $k \ge p+2$, $y_1, \dots, y_k$ have been computed which, when $\eta^4 \le \min\{B_{p+1}, \dots, B_k\}$, satisfies
\begin{align}
\label{sequential lateration_proof2}
\sum_{i \in [k]} \|y_i - x_i\|^2 \le A_k {\eta_k^4},
\end{align}
where 
\begin{align}
\label{sequential lateration_proof3}
{
    A_k
    \le \frac{(1+C_2 \alpha)^{k-p-1}}{(p+1)\underline{\omega}^2}\max\bigg\{C_1\alpha,  \frac{C_2}{(1+C_2\alpha)}\bigg\},
}
\end{align}
and $B_j := A_{j-1}^{-1} (\omega_{j-1}/C_2)^2$ for $p+2 \le j \le k$. ($A_{p+1}$ and $B_{p+1}$ are defined above.)
As we showed above, this is the case when $k = p+2$. We assume this all holds at $k$ and to continue the induction we obtain $y_{k+1}$ by lateration based on $\{y_i : i \in I_k\}$ and the corresponding dissimilarities $\{d_{i, k+1} : i  \in I_k\}$. 
If $\eta$ is small enough that \smash{$\eta^4 \le B_{k+1} := A_k^{-1} (\omega_k/C_2)^{2}$}, in addition to $\eta^4 \le \min\{B_{p+1}, \dots, B_k\}$, from \eqref{sequential lateration_proof2} we have $\nu_{k} \le \omega_{k}/C_2$, so that we may apply \lemref{lateration} to get
{
\begin{align}
\|y_{k+1} - x_{k+1}\|^2 
&\le \frac{C_2}{\omega_{k}^2} \bigg(\rho_{k}^2 \nu_{k}^2 + \frac{\zeta_{k}^4}{|I_{k}|}\bigg) \\
&\le {C_2} \bigg(\alpha A_{k}\eta_{k}^2 + \frac{\zeta_{k}^4}{(p+1)\underline{\omega}^2}\bigg)
\end{align}
implying that 
\begin{align}
\sum_{i \in [k+1]} \|y_i - x_i\|^2
&= \sum_{i \in [k]} \|y_i - x_i\|^2 + \|y_{k+1} - x_{k+1}\|^2 \\
&\le A_{k} \eta_{k}^4 + {C_2} \bigg(\alpha A_{k}\eta_{k}^2 + \frac{\zeta_{k}^4}{(p+1)\underline{\omega}^2}\bigg) \\
&\le (1+C_2\alpha)A_{k} \eta_{k}^4 + \frac{C_2}{(p+1)\underline{\omega}^2}\zeta_{k}^4 \\
&\le A_{k+1} (\eta_{k}^4 + \zeta_{k}^4)\\
&= A_{k+1} \eta_{k+1}^4,
\end{align}
where, using the induction hypothesis \eqref{sequential lateration_proof3},
\begin{align}
A_{k+1} 
&\le \max\bigg\{(1+C_2\alpha)A_{k},  \frac{C_2}{(p+1)\underline{\omega}^2}\bigg\}\\
&= \frac{1}{(p+1)\underline{\omega}^2}\max\bigg\{(1+C_2\alpha) \cdot (1+C_2\alpha)^{k-p-1} \max\Big\{ C_1\alpha, \frac{C_2}{1+C_2\alpha} \Big\},  C_2\bigg\}.\label{sequential lateration_proof_Ak}
\end{align}
By noting that 
$$
(1+C_2\alpha) \cdot \max\Big\{C_1\alpha, \frac{C_2}{1+C_2\alpha}\Big\} = \max\big\{(1+C_2\alpha) C_1\alpha,\; C_2\big\} \ge C_2,
$$
the maximum in \eqref{sequential lateration_proof_Ak} is attained by first argument, and it follows that
\begin{align}
    A_{k+1} \le \frac{(1+C_2 \alpha)^{k-p}}{(p+1)\underline{\omega}^2}\max\bigg\{C_1\alpha,  \frac{C_2}{(1+C_2\alpha)}\bigg\}
\end{align}
We are thus able to proceed with the induction.  At the end of the induction, when all the $y_i$ have been embedded, we obtain the following bound on the accuracy
\begin{align}
\sum_{i \in [n]} \|y_i - x_i\|^2 \le A_n \eta_n^4 = A_n \eta^4,
\end{align}
with
\begin{align}
A_n 
\le \frac{(1+C_2 \alpha)^{n-p-1}}{(p+1)\underline{\omega}^2}\max\bigg\{C_1\alpha,  \frac{C_2}{(1+C_2\alpha)}\bigg\}, \label{A_bound}
\end{align}
whenever $\eta^4 \le B_k$ for $p+1 \le k \le n$, which simplifies to 
\begin{align}
\eta^4 
&\le \min\{B_{p+1}, B_n\} \\
&\le \min\bigg\{(p+1)^2\Big(\frac{\underline\omega}{C_1}\Big)^4, A_{n-1}^{-1} \Big(\frac{\underline\omega}{C_2}\Big)^2\bigg\} \\
&\le \min\Bigg\{\frac{(p+1)^2}{C_1^{4}}, \frac{p+1}{C_2^2(1+C_2\alpha)^{n-p-2} \max\{C_1\alpha,  {C_2}/{(1+C_2\alpha)}\}}\Bigg\} \: \underline\omega^4.\label{sigma_bound}
\end{align}
}
Since the diameter and width of a point set are continuous functions of its pairwise distances, $\alpha$ and $\underline\omega$ are continuous functions of $(\|x_i-x_j\| : (i,j) \in \cE)$,  and this allows us to conclude.
\end{proof}

The bound in \thmref{sequential lateration} takes into account the worst case scenario where, for every $k > p+1$, each node $v_k$ is connected to precisely the preceding $p+1$ nodes $\{v_{k-1}, \dots, v_{k-{p-1}}\}$. In this case, embedding $v_{k+1}$ can only be done by applying classical lateration after the embedding of $v_k$ has been accomplished. This is the most restrictive scenario, which necessitates performing $n-p-1$ rounds of classical lateration to embed all the nodes, and is reflected in the multiplicative growth in the constant $A_n$ in \eqref{A_bound}. The accumulation of error can be avoided in some situations by minimizing the number of rounds of classical lateration, by whenever possible embedding multiple nodes based on the same set of landmark (i.e., already embedded) nodes. The best case scenario corresponds to the landmark MDS setting of \cite{de2004sparse}, where, for every $k > p+1$, each node $v_k$ is connected to the same fixed (and well-conditioned) set $\{v_1, \dots, v_{p+1}\}$.

\section{Perturbation bound for stress minimization}
\label{sec:stress minimization}

In the noisy realizable setting \eqref{setting}, the stress clearly functions as a proxy for the {\em noiseless stress}, defined as
\begin{align}
\label{noiseless}
\sum_{(i,j) \in \cE} \big(\|y_i-y_j\|^2 - \|x_i-x_j\|^2\big)^2.
\end{align}
In turn, the noiseless stress functions as a proxy for the {\em complete noiseless stress}, defined as 
\begin{align}
\label{complete}
\sum_{1 \le i < j \le n} \big(\|y_i-y_j\|^2 - \|x_i-x_j\|^2\big)^2.
\end{align}
We establish below that, in some actionable sense, the stress tracks the complete noiseless stress in the context of a lateration graph. 

\subsection{Rigidity theory}
To investigate this, we turn to {\em rigidity theory}, which examines the question of uniqueness (up to a rigid transformation) when realizing a weighted graph in a given Euclidean space~\cite{thorpe1999rigidity, asimow1978rigidity, laman1970graphs}. 
We introduce some vocabulary from that literature (in particular, from~\cite{connelly2005generic}). As we have already seen, a configuration is a set of $n$ points in $\bbR^p$ indexed by $[n] = \{1, \dots, n\}$. 
A configuration is {\em generic} if the set of its coordinates do not satisfy any nonzero polynomial equation with integer coefficients. 
We say that two configurations $\by = \{y_1, \dots, y_n\}$ and $\bz = \{z_1, \dots, z_n\}$ are {\em congruent} if there is a rigid transformation $f : \bbR^p \to \bbR^p$ such that $z_i = f(y_i)$ for all $i \in [n]$.
A configuration $\by = \{y_1, \dots, y_n\}$ and a graph $\cG = (\cV = [n], \cE)$, together, form a {\em framework}, denoted $\cG(\by)$.  We say that two frameworks, $\cG(\by)$ and $\cG(\bz)$ are {\em equivalent} if 
\begin{align}
\|y_i - y_j\| = \|z_i - z_j\|, \quad \forall (i,j) \in \cE.
\end{align}
The framework $\cG(\by)$ is said to be {\em globally rigid} if, whenever $\cG(\by)$ and $\cG(\bz)$ are equivalent, then necessarily $\by$ and $\bz$ are congruent. 
The graph $\cG$ is said to be {\em generically globally rigid} if $\cG(\by)$ is globally rigid whenever $\by$ is generic.

The complete noiseless stress \eqref{complete} is exactly zero when $\|y_i-y_j\| = \|x_i-x_j\|$ for all $i < j$, and we know this to be equivalent to $\by = \{y_1, \dots, y_n\}$ and $\bx = \{x_1, \dots, x_n\}$ being congruent. For the noiseless stress \eqref{noiseless}, the same is true if $\cG(\bx)$ is globally rigid. This is by mere definition, and we would like to know when this happens. Also by definition, it happens when $\bx$ is a generic configuration and $\cG$ is {\em generically globally rigid}. 

Generic configurations are `common' in the sense that those configurations that are not generic have zero Lebesgue measure (in $\bbR^{np}$). This is simply because there are countably many polynomials with integer coefficients and each one of these defines a surface (its null set) of zero Lebesgue measure. In particular, if a configuration is drawn iid at random from a density, then the configuration is generic with probability one. {Generic configurations have width $\omega > 0$ and ensure that the bounds in \thmref{sequential lateration}, \lemref{CS}, and \lemref{lateration} are not vacuous.}

The question of whether a graph is generically globally rigid or not, is a delicate question. In the very special but useful case of dimension $p = 2$, Jackson~\&~Jord\'{a}n~\cite{jackson2005connected} have shown that if the graph is 6-vertex connected, meaning that it remains connected even after the removal of any 5 vertices, then the graph is generically globally rigid. 
The situation in dimension $p \ge 3$ is more complex, although some useful results exist; see, e.g.,  \cite{hendrickson1992conditions, anderson2009graphical}. A necessary and sufficient condition exists in terms of the existence of an {\em equilibrium stress matrix}, which for a framework $\cG(\bx)$ is defined as a matrix $\omega = (\omega_{ij})$ satisfying $\sum_{j : (i,j) \in \cE} \omega_{ij} (x_i - x_j) = 0$ for all $i \in [n]$.
To a stress matrix $\omega$, we associate another matrix $\Omega = (\Omega_{ij})$ with $\Omega_{ij} = - \omega_{ij}$ when $i \ne j$, and $\Omega_{ii} = \sum_{j} \omega_{ij}$. 
(If we see $\omega$ as the weight matrix of a graph, then $\Omega$ is the corresponding Laplacian.)
Connelly~\cite{connelly2005generic} and Gortler~~et~al.~\cite{gortler2010characterizing}, together, have shown that if $\cG$ has $n \ge p+2$ nodes and is not the complete graph, and if $\bx$ is a generic configuration, then $\cG(\bx)$ is globally rigid if and only if there is a equilibrium stress matrix $\omega$ with $\rank \Omega = n - p - 1$. Aspnes et al. show that lateration graphs (in dimension $p$) are generically globally rigid \cite[Theorem~8]{aspnes2006theory}.

\subsection{Rigidity theory in the presence of noise}
What we have learned so far is that, if the graph $\cG = (\cV, \cE)$ given in the embedding problem is generically globally rigid, and we are in a realizable situation with an underlying configuration $x_1, \dots, x_n$ that is generic, then the noiseless stress \eqref{noiseless} is minimized exactly where the complete noiseless stress \eqref{complete} is minimized, that is, at all the rigid transformations of the configuration. 
These conditions are fulfilled with high probability by a random geometric graph under additional mild conditions (\secref{rgg}).
But all this does not imply much about the noisy stress \eqref{stress}.

While most of the literature on rigidity theory focuses on the noiseless setting, Anderson~et~al.~\cite{anderson2010formal} consider the question of stability in the presence of noise. They do so in the realizable setting in dimension $p=2$, and in the setting where {\em anchors} are given. (Anchors are points whose position is known.) The graph is generically globally rigid with an underlying generic configuration. With anchors, the configuration is effectively unique, not just up to a rigid transformation. In this context, they show that the distance between the minimizer of the stress \eqref{stress} constrained by the anchors and the underlying configuration is bounded by a constant multiple of the noise amplitude. 
Their analysis in based on the results of Connelly~\cite{connelly2005generic} and Gortler~~et~al.~\cite{gortler2010characterizing} mentioned above.

We prove an analogous result in the present anchor-free setting for an arbitrary embedding dimension. We do so for lateration graphs, which in addition to including important models (\secref{rgg}), allows for a completely different proof based on the perturbation bound just established in \thmref{sequential lateration}. 

\begin{theorem}
\label{thm:stability}
In the context of \secref{setting}, consider a noisy realizable situation as in \eqref{setting} in which the network structure $(\cV, \cE)$ is a lateration graph and the latent configuration is in general position. Then, there are $\sigma > 0$ and $A>0$ such that, if $\sum_{(i,j) \in \cE} \eps_{ij}^2 \le \sigma^2$, any minimizer $y^*_1, \dots, y^*_n$ of the stress \eqref{stress} satisfies
\begin{align}
\label{stability1}
\min_{g} \sum_{i \in [n]} \|y^*_i - g(x_i)\|^2 \le A {\sum_{(i,j) \in \cE} \eps_{ij}^2},
\end{align}
where the minimization is over the rigid group of transformations.
\end{theorem}

Once again, and as is the case in \cite{anderson2010formal}, the constants $\sigma>0$ and $A$ depend on the graph and the latent configuration, namely, on the framework $\cG(\bx)$. 
(As it turns out, the proof below shows that we can use the same $\sigma$ and a small multiple of the constant $A$ of \thmref{sequential lateration}.)

\begin{proof}
We first bound the minimum value of the stress. 
Let $y_1, \dots, y_n$ be the embedding given by sequential lateration. 
Let $\sigma_0(\bx) > 0$ and $A_0(\bx) > 0$ be as in \thmref{sequential lateration}. If $\sigma$ is small enough that $\sigma \le \sigma_0(\bx)$, so that $\sum_{(i,j) \in \cE} \eps_{ij}^2 \le \sigma_0(\bx)^2$, the theorem gives the bound
\begin{align}
\label{stability2}
\sum_{i \in [n]} \|y_i - g_0(x_i)\|^2 \le A_0(\bx) \sum_{(i,j) \in \cE} \eps_{ij}^2, 
\le A_0(\bx) \sigma^2,
\end{align}
for some rigid transformation $g_0$. 
Let $y^*_1, \dots, y^*_n$ be a stress minimizer. Since $x_1, \dots, x_n$ is feasible, it must be the case that the stress achieved by $y^*_1, \dots, y^*_n$ is not larger than than the stress achieved by $x_1, \dots, x_n$, so that
\begin{align}
\label{stability3}
\sum_{(i,j) \in \cE} \big(\|y^*_i - y^*_j\|^2 - d_{ij}^2\big)^2
\le \sum_{(i,j) \in \cE} \big(\|x_i - x_j\|^2 - d_{ij}^2\big)^2
= \sum_{(i,j) \in \cE} \eps_{ij}^2.
\end{align}
Therefore, if we define $\xi_{ij} = d_{ij}^2 - \|y^*_i - y^*_j\|^2$, we have that
\begin{align}
\label{xi eps}
\sum_{(i,j) \in \cE} \xi_{ij}^2
\le \sum_{(i,j) \in \cE} \eps_{ij}^2
\le \sigma^2.
\end{align}

We assume, without loss of generality, that $y^*_1,\dots,y^*_n$ are in general position; if not, by randomly perturbing each $y^*_i$ by a small amount ${\gamma_i \ll (1/4|\cE|)\sum_{(i,j) \in \cE}(\eps_{ij}^2 - \xi_{ij}^2)}$, \eqref{xi eps} still holds for ${y^*_i \leftarrow y^*_i + \gamma_i}$. We are looking at applying \thmref{sequential lateration} with the same dissimilarities $(d_{ij})$ and same graph structure but the configuration $y^*_1, \dots, y^*_n$ instead of the configuration $x_1, \dots, x_n$. We may do that if $\sigma$ is small enough that $\sigma \le \sigma_0(\by^*)$, as in that case $\sum_{(i,j) \in \cE} \xi_{ij}^2 \le \sigma_0(\by^*)^2$.
By the fact that $\sigma_0$ and $A_0$ are continuous functions of the noiseless pairwise distances, and the fact that the pairwise distances are continuous functions of the configuration, as can be seen for example via
\begin{align}
    \big| \|y_i - y_j\| - \|x_i - x_j\| \big| 
    &= \big| \|y_i - y_j\| - \|g_0(x_i) - g_0(x_j)\| \big|\\
    &\le \|y_i-g_0(x_i)\| + \|y_j-g_0(x_j)\| \le C(A_0(\mathbf{x}), \sigma_0(\mathbf{x})),
\end{align}
by \eqref{stability2}, if $\sigma$ is small enough (in a way that depends on $\bx$ but not on $y^*$), $\sigma_0(\by^*) \ge \tfrac12 \sigma_0(\bx)$ and $A_0(\by^*) \le 2 A_0(\bx)$, and when this is the case, it suffices that $\sigma \le \tfrac12 \sigma_0(\bx)$ to be able to proceed, and obtain that 
\begin{align}
\label{stability4}
\sum_{i \in [n]} \|y_i - g_1(y^*_i)\|^2 \le A_0(\by^*) \sum_{(i,j) \in \cE} \xi_{ij}^2 \le 2 A_0(\bx) \sum_{(i,j) \in \cE} \eps_{ij}^2,
\end{align}
for some rigid transformation $g_1$. 

Therefore, assuming $\sigma$ is small enough, combining \eqref{stability2} and \eqref{stability4}, together with \eqref{xi eps}, and using the triangle inequality, yields
\begin{align}
\sum_{i \in [n]} \|g_1(y^*_i) - g_0(x_i)\|^2 
&\le 2 \sum_{i \in [n]} \|g_1(y^*_i) - y_i\|^2 + 2 \sum_{i \in [n]} \|y_i - g_0(x_i)\|^2 \\
&\le 4 A_0(\bx) \sum_{i \in [n]} \eps_{ij}^2 + 2 A_0(\bx) \sum_{(i,j) \in \cE} \eps_{ij}^2 \\
&\le 6 A_0(\bx) \sum_{(i,j) \in \cE} \eps_{ij}^2.
\end{align}
We conclude the proof by observing that $\|g_1(y^*_i) - g_0(x_i)\| = \|y^*_i - g_2(x_i)\|$ with $g_2 := g_1^{-1} \circ g_0$ being a rigid transformation.
\end{proof}

While the constant was obtained in a worst-case scenario and is unlikely to be tight in more friendly (and less contrived) settings, it is natural to ask whether the perturbation bound in \thmref{stability} has the right dependency in the errors $\eps_{ij}$.  To this end, the following result answers that question in the affirmative: Any stress minimizer must incur an error that is at least a constant multiple of the noise amplitude.

\begin{proposition}
    Consider the same setting as in \thmref{stability}. Then, for every lateration graph  with network structure $(\cV, \cE)$ and $x_1, \dots, x_n$ in general position, there exists $a > 0$ and dissimilarities $\{d_{ij}: (i, j) \in \cE\}$ in the noisy realizable situation as in \eqref{setting} satisfying $\sum_{(i,j) \in \cE} \eps_{ij}^2 \le \sigma^2$, such that any minimizer $y_1^*, \dots, y_n^*$ of the stress \eqref{stress} satisfies
\begin{align}
\label{lower1}
\min_{g} \sum_{i \in [n]} \|y^*_i - g(x_i)\|^2 \ge a {\sum_{(i,j) \in \cE} \eps_{ij}^2},
\end{align}
where the minimization is over the rigid group of transformations and $a > 0$ is a constant depending only on the graph $(\cV, \cE)$.
\end{proposition}

\begin{proof}
    Let $(\cV, \cE)$ be a lateration graph and let $\delta := \max_{i \in [n]}\textup{deg}(v_i)$ be the maximum degree of the nodes in $\cV$. For $0 < \eta \le 1$, let
    \begin{align}
        d_{ij}^2 = \eta^2\|x_i - x_j\|^2\qquad (i, j) \in [n]^2.
    \end{align}
    and set $y^*_1, \dots, y^*_n$ to be $y^*_i := \eta x_i$ for each $i \in [n]$. Then, it is straightforward to verify that $\|y^*_i - y^*_j\|^2 = d_{ij}^2$ for all $(i, j) \in \cE$, and therefore, $y^*_1, \dots, y^*_n$ is a minimizer of the stress \eqref{stress}. Moreover, by construction, since $x_i \mapsto y^*_i$ is a scale transformation, and the optimal rigid transformation $g$ is the identity map; therefore
    \begin{align}\label{lower-bound1}
        \min_g \sum_{i \in [n]} \|y^*_i - g(x_i)\|^2 = (1-\eta)^2\sum_{i \in [n]}\|x_i\|^2.
    \end{align}
    For $\eps_{ij}^2 = (d_{ij}^2 - \|x_i-x_j\|^2)^2$, we have
    \begin{align}
        \sum_{(i,j) \in \cE} \eps_{ij}^2
        &= \sum_{(i,j) \in \cE} (\eta^2-1)^2 \cdot \|x_i - x_j\|^2 \label{lower-bound21} \\
        &\le (1+\eta)^2(1-\eta)^2 \cdot 2\sum_{(i,j) \in \cE} \big(\|x_i\|^2 + \|x_j\|^2\big) \\
        &\le 8(1-\eta)^2 \delta \sum_{i \in [n]} \|x_i\|^2,
        \label{lower-bound2}
    \end{align}
    where the first inequality uses $\|x_i-x_j\|^2 \le 2(\|x_i\|^2 + \|x_j\|^2)$, and the second inequality follows from the fact that $(1+\eta)\le 2$ for $0 < \eta \le 1$, and 
    \begin{align}
        \sum_{(i,j) \in \cE}(\|x_i\|^2 + \|x_j\|^2) = \sum_{i \in [n]}\text{deg}(v_i) \|x_i\|^2 \le \delta \sum_{i \in [n]} \|x_i\|^2.
    \end{align}
    Combining \eqref{lower-bound1} and \eqref{lower-bound2}, for $a = 1/8\delta$ we get
    \begin{align}
        \min_g \sum_{i \in [n]} \|y^*_i - g(x_i)\|^2 \ge a \sum_{(i,j) \in \cE} \eps_{ij}^2,
    \end{align}
    which completes the proof. Moreover, when 
    $$
    1 - \frac{\sigma}{\sqrt{ \sum_{(i,j) \in \cE} \|x_i - x_j\|^2 }} \le \eta^2 \le 1,
    $$
    it follows from \eqref{lower-bound21} that $\sum_{(i,j) \in \cE} \eps_{ij}^2 \le \sigma^2$.
\end{proof}

\section{Random geometric graphs}
\label{sec:rgg}

In the literature, the main stochastic model is a {\em random geometric graph}~\cite{penrose2003random}. Such a graph has node set representing points that are drawn iid from some distribution on $\bbR^p$ and edges between any two of these points within distance $r$. For example, Aspnes~et~al.~\cite{aspnes2006theory} show that, for the uniform distribution on $[0,1]^2$, as the size of the configuration increases, if the connectivity radius is not too small, the probability that the resulting graph is generically globally rigid, and that the corresponding framework is globally rigid, tends to one. We generalize their result. 

\begin{theorem}
\label{thm:rgg}
Suppose a configuration of cardinality $n$ is drawn iid from a distribution $P$ with $\textup{supp}(P) = \bar\Omega \subset \R^p$ where $\Omega$ is bounded, open, and connected. Considering the asymptotic regime $n\to\infty$, there is $r_n \to 0$ such that a graph built on this configuration with a connectivity radius $r \ge r_n$ is a lateration graph with probability tending to one.  
\end{theorem}

The conditions on the support of the distribution generating the locations of the sensors are very mild. The requirement for $\text{supp}(P)$ to be the closure of an bounded, open set $\Omega$ is satisfied, for instance, by any distribution which admits a  enforces the required genericity of the configuration, since this disallows the points from being sampled from an algebraic curve in $\R^p$. We could even relax the condition that $\Omega$ is connected as long as the connectivity radius $r$ exceeds the maximum separation between its connected components. 

\begin{proof}
Let $\cG_r(\bx)$ be the neighborhood graph with connectivity radius $r$ built on the point set $\bx = \{x_1, \dots, x_n\}$. It is obvious that the property of being a lateration graph is monotonic in $r$ in the sense that if $\cG_r(\bx)$ is a lateration graph then so is $\cG_s(\bx)$ for any $s > r$. It therefore suffices to find $r_n \to 0$ such that $\cG_{r_n}(x_1, \dots, x_n)$ is lateration graph with probability tending to 1. (All limits are as $n\to\infty$ unless otherwise specified.)

$\Omega$ being bounded, for any $m \ge 1$ integer, it can be covered with finitely many, say $N_m$, open balls of radius $1/2m$ centered on points belonging to $\Omega$. 
(We even know that the minimum number $N_m$ satisfies $N_m \le C_0 m^p$, where $C_0$ depends on $\diam(\Omega)$ and $p$, although this will not play a role in what follows.)
We consider such a covering, with balls denoted $B^m_1, \dots, B^m_{N_m}$. Let $A_j := B_j \cap \Omega \ne \emptyset$ for all $j$.
Form the following graph: the node set is $A^m_1, \dots, A^m_{N_m}$, and $A^m_j$ and $A^m_k$ are connected if they intersect. We call this the {\em cover graph}.
Because $\Omega$ is connected, the cover graph must also be connected, and may therefore be traversed by, say, depth-first search, which starting at any $A^m_{j_0}$ results in a (finite) path in the cover graph that passes through the entire graph, meaning, a sequence $(A^m_{j_s}: s = 0, \dots, S_m)$ with $I^m_s := A^m_{j_{s-1}} \cap A^m_{j_s} \ne \emptyset$ for all $s$, with the property that, for any $j$, there is $s$ such that $A^m_{j_s} = A^m_j$.
Note that, by construction, each $I^m_s$ is a nonempty open subset of $\Omega$ of diameter $<1/m$; together, these subsets cover $\Omega$. (Note that some of these sets might coincide, but this is unimportant.)

Now, let $x_1, \dots, x_n$ denote an iid sample from $P$, and for a Borel set $A$, 
{let $P(A) = P(\{x_i \in A\})$.}
Consider the event $E^{n,m}_s$ that $I^m_s$ contains at least $p+1$ sample points, and define $E^{n,m} = \cap_s E^{n,m}_s$, which is the event that each one of the subsets $I^m_1, \dots, I^m_{S_m}$ contains at least $p+1$ sample points. Let $a(n,m) = 1 - \bbP(E^{n,m})$, which is the probability that $E^{n,m}$ fails to happen. Note that, essentially by definition, $a(n,m)$ is decreasing in $n$. In addition to that, we also have $\lim_{n\to\infty} a(n,m) = 0$. To see this, we derive, by the union bound and the fact that the number of points falling in a Borel set $A$ is binomial with parameters $n$ and $P(A)$,
\begin{align}
a(n,m)
&\le \sum_{s=0}^{S_m} \big(1 - \bbP(E^{n,m}_s)\big) \\
&\le \sum_{s=0}^{S_m} \sum_{k=0}^p \binom{n}{k} P(I^m_s)^k (1-P(I^m_s))^{n-k} \\
&\le (S_m+1) (p+1) n^p (1-b_m)^{n-p}, \qquad
b_m := \min_{s=0,\dots,S_m} P(I^m_s).
\end{align}
Since each $I^m_s$ is a nonempty open subset of $\Omega$, we have that $b_m > 0$, and so $a(n,m) \to 0$ as $n\to\infty$ when $m$ remains fixed. (The convergence is exponentially fast, although this will not play a role.)
The fact that $a(n,m)$ is decreasing in $n$ and $\lim_{n\to\infty} a(n,m) = 0$ implies, via elementary arguments, that there is sequence $m_n \to \infty$ such that $\lim_{n\to\infty} a(n,m_n) = 0$, or equivalently, $\bbP(E^{n,m_n}) \to 1$ as $n\to\infty$. 

We now prove that, under $E^{n,m}$, the neighborhood graph built on the sample points $x_1, \dots, x_n$ with connectivity radius $r=1/m$ is a lateration graph. Thus, we work under the situation where each $I^m_s$ contains at least $p+1$ sample points. 
First, consider $p+1$ such points in $I^m_1$, and label them $v_1, \dots, v_{p+1}$ in any order. Since $\diam(I^m_1) < r$, the subgraph that these points induce is complete. Recall that $I^m_1 \subset A^m_{j_0}$.   Label the remaining points in $A^m_{j_0}$ as $v_{p+1}, \dots, v_{n_0}$ and note that, since $\diam(A^m_{j_0}) < r$, each of these points is connected to all the points $v_1, \dots, v_{p+1}$. Let $\cV_0$ denote $\{v_1, \dots, v_{n_0}\}$. 
Similarly, recall that $I^m_1 \subset A^m_{j_1}$; label the remaining points in $A^m_{j_1}$ as $v_{n_0+1}, \dots, v_{n_1}$, and note that, since $\diam(A^m_{j_1}) < r$, each of these points is connected to all the points $v_1, \dots, v_{p+1}$, and therefore to at least $p+1$ points inside $\cV_0$; let $\cV_1 := \{v_1, \dots, v_{n_1}\}$.
Suppose that we are at a stage where we have built an ordering $\cV_{s-1} = \{v_1, \dots, v_{n_{s-1}}\}$ of the sample points in $A^m_{j_{0}}, \dots, A^m_{j_{s-1}}$ such that, for each $p+1 < j \le j_{s-1}$, $v_j$ is connected to at least $p+1$ points among $v_1, \dots, v_{j-1}$. In particular, this includes all the points in $I^m_s$ since $I^m_s \subset A^m_{j_{s-1}}$. Now, $I^m_s \subset A^m_{j_s}$ also; label the remaining points in $A^m_{j_s}$ as $v_{n_{s-1}+1}, \dots, v_{n_s}$, and since $\diam(A^m_{j_s}) < r$, each of these points is connected to all the points $I^m_s$. Since $I^m_s$ contains at least $p+1$ points (because $E^{n,m}$ holds), we may continue the recursion by letting $\cV_s = \{v_1, \dots, v_{n_{s}}\}$. Doing so until all the sample points have been processed provides a laterative ordering of the entire neighborhood graph $\cG_r(x_1, \dots, x_n)$.
\end{proof}

\begin{figure}[t!]
    \centering
    \includegraphics[width=0.8\textwidth]{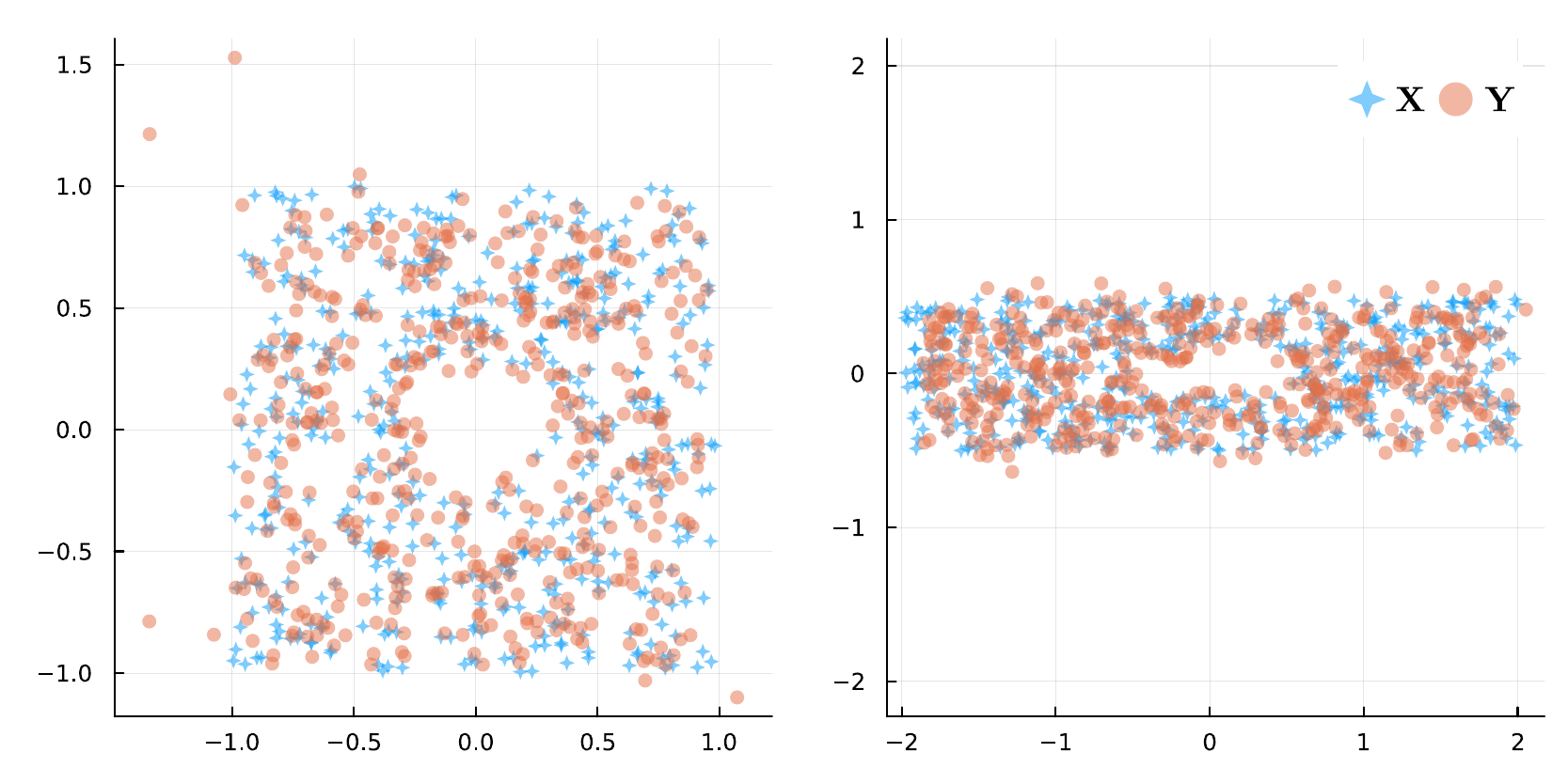}
    \caption{Examples of latent configurations $x_1, x_2, \dots, x_n$ and the embedding $y_1, y_2, \dots, y_n$ obtained from sequential lateration when (left) $h=0.5$ and $\kappa=1$, and when (right) $h=0.5$ and $\kappa=2$. The model is \eqref{setting}, with $\eps_{ij} \sim N(0, \varsigma^2)$ for $\varsigma^2=0.1$.}
    \label{fig:example}
    \end{figure}
    
\begin{figure}[ht!]
    \centering
    \begin{subfigure}{0.32\textwidth}
        \caption{Connectivity radius $r$}
        \includegraphics[width=\textwidth]{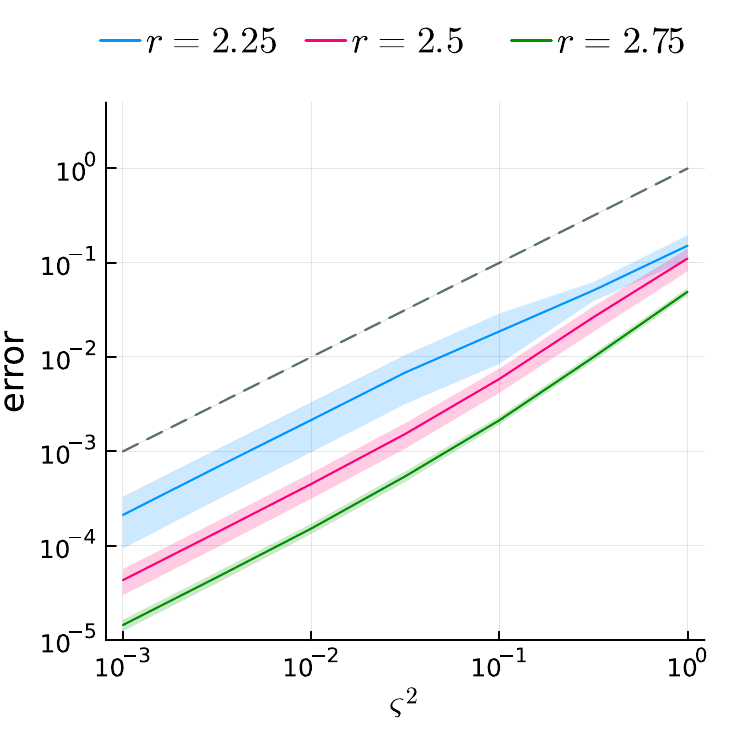}
        \label{subfig:p2}
    \end{subfigure}
    \begin{subfigure}{0.32\textwidth}
        \caption{Aspect-ratio $\kappa$}
        \includegraphics[width=\textwidth]{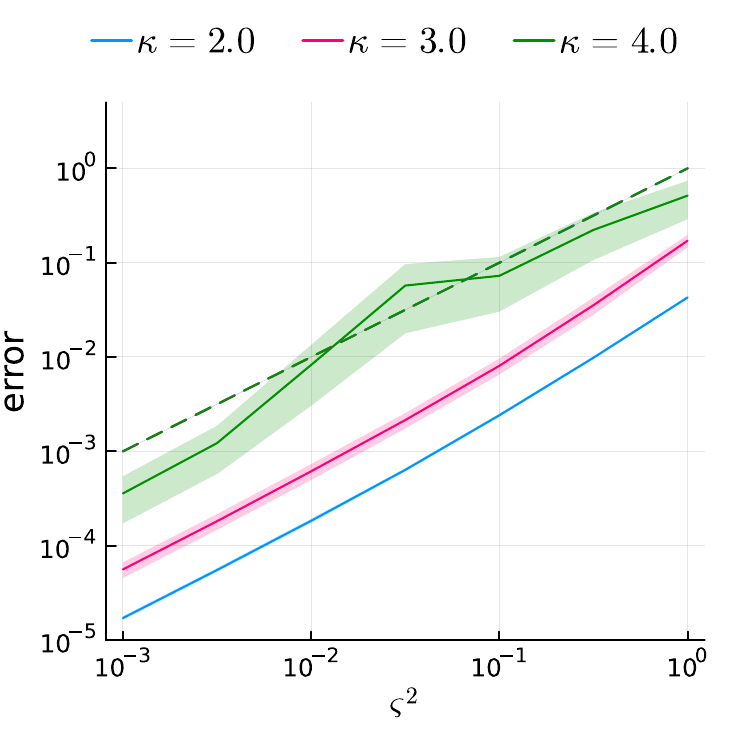}
        \label{subfig:p3}
    \end{subfigure}
    \begin{subfigure}{0.32\textwidth}
        \caption{Hollowing out $h^2$}
        \includegraphics[width=\textwidth]{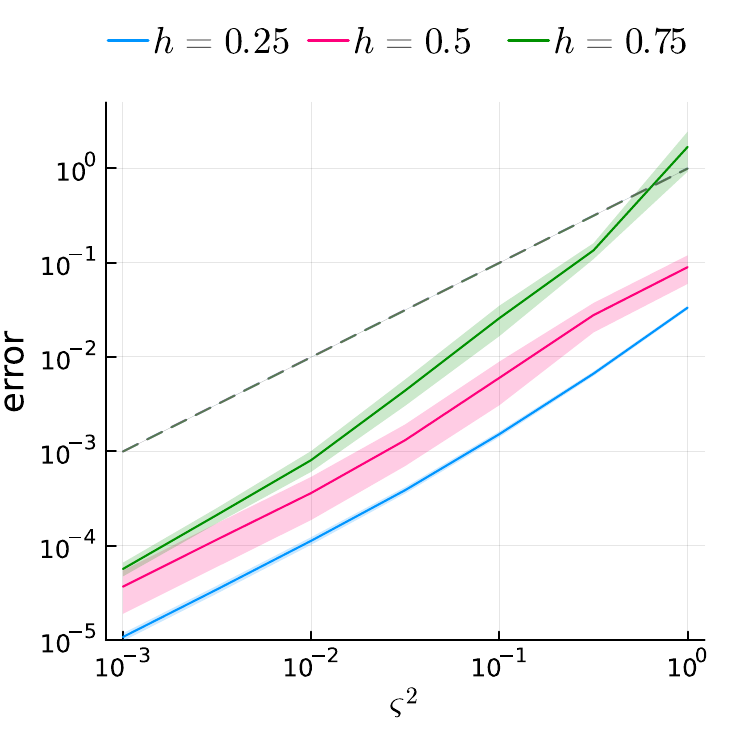}
        \label{subfig:p4}
    \end{subfigure}
    \caption{Results of the numerical experiments. The vertical axis in all plots is the embedding error and the horizontal axis is the variance of the noise, $\varsigma^2$. The results are shown on a log-log scale. {The dashed line in (a),~(b)~and~(c) is the $45^\circ$ line corresponding to the mean perturbation $s(\eps)^2$ defined in \eqref{s}.}}
    \label{fig:experiments}
\end{figure}

\begin{figure}[t!]
    \centering
    \begin{subfigure}{0.37\textwidth}
        \centering
        \caption{Comparison of performance}
        \includegraphics[width=\textwidth]{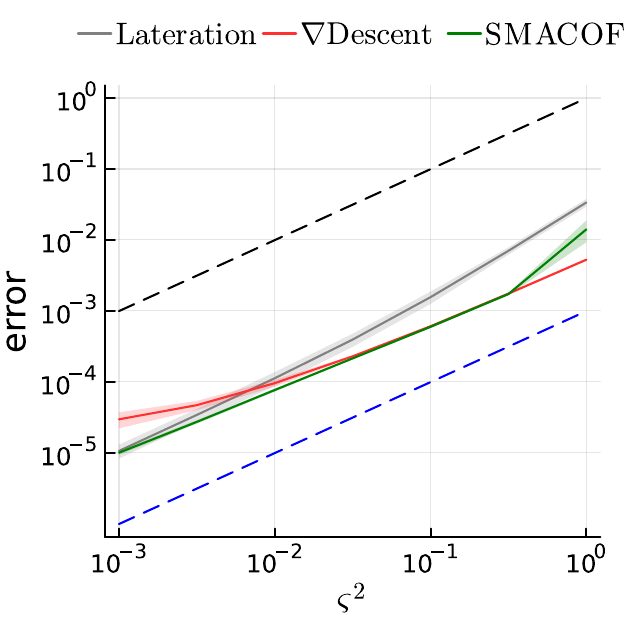}
        \label{subfig:comparison}
    \end{subfigure}
    \quad\quad
    \begin{subfigure}{0.37\textwidth}
        \centering
        \caption{Computational time}
        \includegraphics[width=\textwidth]{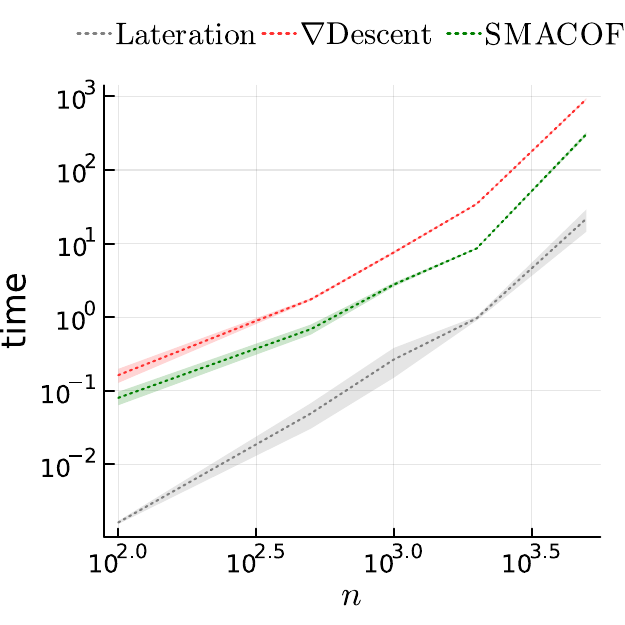}
        \label{subfig:time}
    \end{subfigure}
    \caption{(a) Comparison of the embedding error for SMACOF and Gradient Descent shown on a log-log scale. The black dashed line corresponds the mean perturbation, $s(\eps)^2$, and the blue dashed line is a plot of $s(\eps)^2/10^3$ which provides evidence of a lower bound for the embedding error. (b) Computational time for sequential lateration, SMACOF and Gradient Descent for varying sample sizes $n$.}
    \label{fig:comparison}
\end{figure}

\section{Numerical experiments}
\label{sec:numerics}

We probe the accuracy of the stability bound in \thmref{stability} in the following numerical experiments. We begin by noting that the constants $\sigma, A > 0$ in  \thmref{stability} depend on the graph $\cG$ and the latent configuration ${x_1, x_2, \dots, x_n}$. In particular, for a fixed graph $\cG$ the constant $A$ depends on the aspect-ratio $(\rho/\omega)^2$. 

Therefore, in order to investigate the stability bound, we consider the setting where the latent configuration $x_1, x_2, \dots, x_n$ is drawn \textit{iid} from a uniform distribution on the domain $\Omega(h, \kappa)$, where for $h \in (0, 1)$ and scale $\kappa > 0$, 
\[\Omega(h, \kappa) := [-\kappa, \kappa] \times [-\kappa^{-1}, \kappa^{-1}] \setminus [-h\kappa, h\kappa] \times [-h\kappa^{-1}, h\kappa^{-1}],\]
is a rectangle with aspect ratio $\kappa^2 \in (0, 1)$ and a fraction $h^2 \in (0, 1)$ of its area hollowed out from the center. The parameters $h$ and $\kappa$ together account for the complexity of the latent configuration. 

We consider the setting where the dissimilarities are corrupted by additive noise $\eps_{ij}$, i.e., $d_{ij}^2 = \max\{\|x_i - x_j\|^2 + \eps_{ij}, 0\}$, where $\eps_{ij}$ are drawn \textit{iid} from $N(0, \varsigma^2)$. 
See Figure~\ref{fig:example} for an illustration. {Given the graph $\cG$ with dissimilarities $d_{ij}$, we obtain the embedding $y_1, y_2, \dots, y_n$ using the `first' sequential lateration method {described earlier in the paper}. Specifically, we find a suitably large clique in the graph via greedy search and embed it via  classical scaling; we then embed the remaining nodes recursively by classical lateration and compute the embedding error}
$$
\frac{1}{n}\sum_{i \in [n]}\| y_i - \widehat{g}(x_i)\|^2,
$$
where the rigid transformation $\widehat{g}: \R^{p} \to \R^p$ is obtained via Procrustes alignment \cite{arias2020perturbation}. In all experiments we compare the embedding error of $y_1, y_2, \dots, y_n$ to the mean perturbation,
\begin{equation}
\label{s}
s(\eps)^2 := \frac{1}{|\cE|} \sum_{(i,j) \in \cE} \eps_{ij}^2,
\end{equation}
which is the normalized bound on the right hand side of \eqref{lateration1} in \thmref{sequential lateration}. Note that, when $|\cE|$ is large, $s(\eps)^2 \approx \mathbb{E}(\eps_{ij}^2) = \varsigma^2$ by the law of large numbers.

\medskip

The results are summarized in Figure~\ref{fig:experiments}. For fixed $n=500, h=0.2$ and $\kappa=1$, Figure~\ref{subfig:p2} shows the effect of the effect of the connectivity radius ($r \in \{2.25, 2.5, 2.75\}$) of the random geometric graph on the accuracy of the bound. Figure~\ref{subfig:p3} illustrates the effect of the aspect ratio ($\kappa \in \{2, 3, 4\}$) at a fixed connectivity radius of $r=0.3$. Lastly, for fixed $\kappa=1$ and $r=0.3$, Figure~\ref{subfig:p4} shows the effect of the hollowing out ($h \in \{0.25, 0.5, 0.75\}$) of the domain of the latent configuration. In all cases, the results corroborate the bound established in \thmref{sequential lateration}. Furthermore, as seen in the plots, the constants are likely larger for more complex latent configurations, i.e., when $r$ is small, $\kappa$ is small, or when $h$ is large.

Figure~\ref{fig:comparison} investigates the accuracy of the bound established in \thmref{stability}. We compare the embedding error of the sequential lateration method to the embedding error from stress-minimizers, $y_1^*, y_2^*, \dots, y_n^*$, obtained using (i) gradient descent, and (ii) the SMACOF algorithm \citep{de2009multidimensional}. Figure~\ref{subfig:comparison} shows the embedding error of the three methods compared to the mean perturbation $s(\eps)^2$ (the black dashed line), and confirms the bound in \thmref{stability}. While the embedding error of $y_1^*, y_2^*, \dots, y_n^*$ marginally improves on the embedding error of sequential lateration, the advantage of the sequential lateration procedure is the reduced computational time which, as shown in Figure~\ref{subfig:time}, is between one to two orders of magnitude faster than SMACOF and gradient descent.

\section{Discussion}
\label{sec:discussion}

Our main contribution in this paper is a perturbation bound for sequential lateration. This provides a way to understand and, to some extent, quantify the stability of sequential lateration in the presence of noise. As a corollary, we obtained a perturbation bound for stress minimization in the setting of a lateration graph. As we mentioned earlier, this addresses the issue of noise in multidimensional scaling / network localization discussed and formulated as a set of open questions by Mao~et~al.~\cite{mao2007wireless} in their well-known review paper.

The bounds stated in \thmref{sequential lateration} and \thmref{stability} are not explicit, and at least when in comes to the stability of sequential lateration, it might be of interest to obtain a more explicit bound. In fact, we did this in the proof of \thmref{sequential lateration} --- see \eqref{A_bound} and \eqref{sigma_bound}. However, our analysis can only be accurate, if at all, in the very worst case, and is not representative of situations such as that of a random geometric graph (\secref{rgg}) where it is often possible to embed many nodes based on the same already embedded nodes, and doing this as much as possible avoids the accumulation of error. We also note that, to remove noise, one can do better by embedding by classical scaling many nodes at once, not just the minimum required number of $p+1$ nodes.

While we have focused on stability to noise, a related but distinct issue is the presence of outliers, by which we mean gross errors (i.e., some of the error terms $\eps_{ij}$ in \eqref{setting} could be quite large). There are robust methods\footnote{We use the term `robust' in the way it is used in statistics. In the broad MDS literature, this term is sometimes used to mean what we refer to here as stability to noise.} for MDS, e.g., \cite{heiser1988multidimensional, cayton2006robust}, but their robustness properties are not well-understood. Converting the available (metric) data into ordinal data, by replacing $d_{ij}$ by its rank among all dissimilarities $(d_{kl})_{(k,l) \in \cE}$, and then applying a method for ordinal MDS is likely to yield a robust method, but the robustness of such methods are also poorly understood. For some effort in this direction, see \cite{jain2016finite}.

\subsection*{Acknowledgements} 
This work was partially supported by the US National Science Foundation (DMS 1916071).

\small
\bibliographystyle{chicago}
\bibliography{../ref2}

\end{document}

%% file: notation.tex
\newcommand{\R}{\mathbb{R}}

\newtheorem{theorem}{Theorem}[section]
\newtheorem{proposition}[theorem]{Proposition}
\newtheorem{lemma}[theorem]{Lemma}

\theoremstyle{definition}

\theoremstyle{remark}

\numberwithin{equation}{section}
\numberwithin{figure}{section}

\def\beq{\begin{equation}} 
\def\eeq{\end{equation}}
\def\beqn{\begin{eqnarray*}}
\def\eeqn{\end{eqnarray*}}
\def\Bitem{\begin{itemize}\setlength{\itemsep}{.2in}}
\def\bitem{\begin{itemize}\setlength{\itemsep}{.05in}}
\def\eitem{\end{itemize}}
\def\Benum{\begin{enumerate}\setlength{\itemsep}{.2in}}
\def\benum{\begin{enumerate}\setlength{\itemsep}{.05in}}
\def\eenum{\end{enumerate}}
\def\bmult{\begin{multline*}}
\def\emult{\end{multline*}}
\def\bcenter{\begin{center}}
\def\ecenter{\end{center}}
\def\bframe{\begin{frame}}
\def\eframe{\end{frame}}

\newcommand{\thmref}[1]{Theorem~\ref{thm:#1}}

\newcommand{\lemref}[1]{Lemma~\ref{lem:#1}}
\newcommand{\secref}[1]{Section~\ref{sec:#1}}

\DeclareMathOperator{\diam}{diam}

\DeclareMathOperator{\rank}{rank}

\def\cE{\mathcal{E}}

\def\cG{\mathcal{G}}

\def\cV{\mathcal{V}}

\def\bx{\mathbf{x}}
\def\by{\mathbf{y}}
\def\bz{\mathbf{z}}

\def\bbP{\mathbb{P}}

\def\bbR{\mathbb{R}}

\def\eps{\varepsilon}

\def\1{\mathbbm{1}}
